\documentclass[10pt,a4paper]{amsart}
\usepackage[hidelinks]{hyperref}
\usepackage{amssymb,amsthm}
\usepackage[all]{xy}

\title[On the symplectic eightfold associated to a Pfaffian cubic fourfold]{On the symplectic eightfold \\ associated to a Pfaffian cubic fourfold}
\author{N.~Addington}
\email{adding@math.duke.edu}
\address{\noindent Department of Mathematics \\
Duke University, Box 90320 \\
Durham, NC 27708-0320 \\
United States}
\author{M.~Lehn}
\email{lehn@mathematik.uni-mainz.de}
\address{Institut f\"ur Mathematik \\
Johannes Gutenberg--Universit\"at Mainz \\
55099 Mainz \\
Germany}

\newtheoremstyle{ourthm}
{3pt}
{3pt}
{\slshape}
{}
{\bf\upshape}
{\ ---}
{.5em}
{}
\theoremstyle{ourthm}
\newtheorem*{theorem*}{Theorem}
\newtheorem*{corollary*}{Corollary}
\newtheorem{proposition}{Proposition}
\newtheorem{lemma}[proposition]{Lemma}

\newcommand \ka {\mathcal A}
\newcommand \kc {\mathcal C}
\newcommand \kh {\mathcal H}
\newcommand \ko {\mathcal O}
\newcommand \kp {\mathcal P}
\newcommand \kq {\mathcal Q}
\newcommand \IP {\mathbb P}

\let \phi \undefined

\newcommand \dual {\makebox[0mm]{}^{{\scriptstyle\vee}}}
\newcommand \ddual {\dual\dual}
\DeclareMathOperator \cone {cone}
\DeclareMathOperator \codim {codim}
\DeclareMathOperator \Hom {Hom}
\DeclareMathOperator \Ext {Ext}
\DeclareMathOperator \pf {pf}
\DeclareMathOperator \rk {rk}
\DeclareMathOperator \Hilb {Hilb}
\DeclareMathOperator \Grass {Grass}
\DeclareMathOperator \Coh {Coh}
\DeclareMathOperator \pr {pr}
\DeclareMathOperator \rad {rad}
\DeclareMathOperator \RHom {RHom}

\newcommand \lra \longrightarrow
\newcommand \xra \xrightarrow
\newcommand \xla \xleftarrow

\begin{document}

\begin{abstract}
We show that the irreducible holomorphic symplectic eightfold $Z$ associated to a cubic fourfold $Y$ not containing a plane is deformation-equivalent to the Hilbert scheme of four points on a K3 surface.  We do this by constructing for a generic Pfaffian cubic $Y$ a birational map $Z \dashrightarrow \Hilb^4(X)$, where $X$ is the K3 surface associated to $Y$ by Beauville and Donagi.  We interpret $Z$ as a moduli space of complexes on $X$ and observe that at some point of $Z$, hence on a Zariski open subset, the complex is just the ideal sheaf of four points.  This note is an appendix to \cite{llsvs}.
\end{abstract}

\maketitle

\section*{Introduction}

Beauville and Donagi \cite{bd} showed that if $Y \subset \IP^5$ is a smooth cubic hypersurface then the variety $F$ of lines on $Y$ is an irreducible holomorphic symplectic fourfold.  They did this by showing that for certain special cubics, called Pfaffian cubics, there is an associated K3 surface $X$ such that $F \cong \Hilb^2(X)$.  Kuznetsov later observed that for a general $Y$, the K3 surface $X$ can be replaced with a ``K3 category'' $\ka$, and he and Markushevich showed that $F$ is a moduli space of objects in $\ka$, which in some sense explains the symplectic form on $F$.  In more detail, the derived category $D(Y) = D^b(\Coh(Y))$ admits a semi-orthogonal decomposition
\[ D(Y)=\langle \ka,\ko_Y(-1),\ko_Y,\ko_Y(1)\rangle, \]
where $\ka$ is like the derived category of a K3 surface in that it has the same Serre functor and Hochschild homology and cohomology, and $\ka \cong D(X)$ if $Y$ is Pfaffian \cite{kuz_gr}.  Given a line $\ell \subset Y$, the projection of the ideal sheaf $I_\ell$ into $\ka$ is a stable sheaf whose deformation space is naturally identified with that of $\ell$ \cite[\S5]{km} .

Lehn et al.\ \cite{llsvs} associated to each cubic $Y$ not containing a plane an irreducible holomorphic symplectic eightfold $Z$, constructed not from lines but from twisted cubics on $Y$.  They calculated that $Z$ has the same topological Euler number as $\Hilb^4(\text{K3})$, but left open the question of whether the two are deformation equivalent.  In this note, using a derived interpretation like that of Kuznetsov and Markushevich, we show that they are.
\begin{theorem*}
If $Y$ is a Pfaffian cubic fourfold not containing a plane and the associated K3 surface $X$ does not contain a line then $Z$ is birational to $\Hilb^4(X)$.
\end{theorem*}
\begin{corollary*}
For any cubic fourfold $Y$ not containing a plane, $Z$ is deformation equivalent to the Hilbert scheme of four points on a K3 surface.
\end{corollary*}
\noindent The corollary follows from Huybrechts' theorem that birational holomorphic symplectic varieties are deformation equivalent \cite[Thm.~4.6]{huybrechts}.

In Section \ref{sec1} we interpret $Z$ as a moduli space of objects in $\ka$, clarifying the construction of \cite{llsvs}.  There, $Z$ was constructed as a contraction of the moduli space $M$ of generalized twisted cubics on $Y$; precisely, there is an embedding $j\colon Y \to Z$ such that $M$ is a $\IP^2$-bundle over the blow-up of $Z$ along $j(Y)$.  Here we show that two points $[C_1], [C_2] \in M$ lie in the same fiber of $M \to Z$ if and only if the projections of the twisted ideal sheaves $I_{C_1}(2)$ and $I_{C_2}(2)$ into the subcategory $\ka$ are the same, and that if $[C]$ lies over $j(y)$ then the projection of $I_C(2)$ is the same as the projection of the skyscraper sheaf $\ko_y$, up to a shift.

In Section \ref{sec2} we recall Beauville and Donagi's construction of a K3 surface $X$ associated to a Pfaffian cubic $Y$, and give a self-contained proof of Kuznetsov's equivalence $\ka \cong D(X)$, emphasizing its geometric content: it is induced by the ideal sheaf of a correspondence $\Gamma \subset X \times Y$ that is generically 4-to-1 over $Y$.

In Section \ref{sec3} we argue that if $Y$ is Pfaffian then for a general $[C] \in M$, the projection of $I_C(2)$ into $\ka \cong D(X)$ is the ideal sheaf of four points in $X$, again up to a shift.  Rather than proving this directly, we observe that if $[C]$ lies over $j(y)$ then we can instead pass $\ko_y$ over to $D(X)$, and for generic $y$ this clearly yields the ideal sheaf of four points; but the property of being an ideal sheaf is an open condition.  Thus we get a map from a Zariski open subset $M_0 \subset M$ to $\Hilb^4(X)$, and from our work in Section \ref{sec1} we see that this descends to an embedding of an open subset $Z_1 \subset Z$ into $\Hilb^4(X)$.

After the preprint of this paper appeared, Ouchi \cite{ouchi} showed that for a very general $Y$ containing a plane -- that is, in the opposite situation to \cite{llsvs} -- there is a Bridgeland stability condition on $\ka$ under which the projection of the skyscraper sheaf $\ko_y$ into $\ka$ is stable; hence $Y$ embeds as a Lagrangian in a holomorphic symplectic eightfold deformation-equivalent to $\Hilb^4(\text{K3})$.  Presumably his eightfold is the limit of $Z$ as $Y$ approaches the locus of cubics containing a plane.

\subsection*{Acknowledgements}
We thank E.~Macr\`i and P.~Stellari for discussing their related work in progress, A.~Kuz\-net\-sov for discussing his, and J.~Starr for informing us of his related unpublished work.  We started this project at a workshop hosted by Y.~Nami\-kawa at the Research Institute for Mathematical Sciences in Kyoto, Japan, and finished it while the first author was visiting the Hausdorff Research Institute for Mathematics in Bonn, Germany; we thank both for their hospitality.  Our travel was supported by SFB Transregio 45 Bonn-Mainz-Essen and NSF grant no.\ DMS--0905923.

\section{\texorpdfstring{$Z$ as a moduli space of objects in $\ka$}{Z as a moduli space of objects in A}} \label{sec1}

Let $M=\Hilb^{gtc}(Y)$ be the irreducible component of $\Hilb^{3n+1}(Y)$ containing twisted cubics, and let $u\colon M \to Z$ be the contraction and $j\colon Y \to Z$ the embedding that appear in \cite{llsvs}.  Recall from \cite{llsvs} that $u$ factors as $\sigma \circ a$, where $a\colon M \to Z'$ is a $\IP^2$-bundle and $\sigma\colon Z' \to Z$ is the blow-up of $Z$ along $j(Y)$.  The analysis of the curves $C$ parametrized by $M$ breaks into two cases, depending on whether $C$ is arithmetically Cohen-Macaulay (aCM) or non-Cohen-Macaulay (non-CM).

If $u([C]) \notin j(Y)$ then $C$ is aCM.  The linear hull of $C$ is a $\IP^3$, and the ideal sheaf of $C$ in this $\IP^3$ has a resolution of the form
\begin{equation} \label{eq:acm_res}
0 \to \ko_{\IP^3}(-3)^2 \to \ko_{\IP^3}(-2)^3 \to I_{C/\IP^3} \to 0.
\end{equation}
Let $S_C = Y \cap \IP^3$, which is a cubic surface.  Any curve $C'$ corresponding to a point in the same fiber $a^{-1}(a([C]))$ is contained in the same cubic surface $S_C$.  Moreover there is a $3 \times 3$ matrix $A$ with entries in $H^0(S_C, \ko(1))$ such that for all such $C'$ the ideal sheaf $I_{C'/S_C}$ is generated by the minors of a $3 \times 2$ matrix $A_0$ consisting of two independent linear combinations of columns of $A$.  Finally, $I_{C'/S_C}$ admits a 2-periodic resolution
\[ \dotsb \xra{\;A\;}\ko_{S_C}(-5)^3\xra{\;B\;}\ko_{S_C}(-3)^{\oplus 3}\xra{\;A\;} \ko_{S_C}(-2)^3\lra I_{C'/S_C}\lra 0 \]
where $B$ is the adjugate matrix of $A$.  In particular, as abstract sheaves, all $I_{C'/S_C}$ for points $[C']$ in the same $a$-fiber are isomorphic.  The converse holds as well: for any $[C'] \in M$ with $I_{C'/S_{C'}} \cong I_{C/S_C}$ we have $[C'] \in a^{-1}(a([C]))$.  To see this, note that the curve $C$ can be reconstructed from its ideal sheaf by a choice of homomorphism $I_{C/S_C} \to \ko_{S_C}$.  From the resolution \eqref{eq:acm_res} and the exact sequence
\begin{equation} \label{eq:ICS_vs_ICP}
0 \to \ko_{\IP^3}(-3) \to I_{C/\IP^3} \to I_{C/S_C} \to 0
\end{equation}
we find that $\Hom(I_{C/S_C}, \ko_{S_C}) = H^2(I_{C/S_C}(-1))^*$ is 3-dimensional, which gives a $\IP^2$-family of distinct curves with isomorphic $I_{C/S_C}$.  But the fiber $a^{-1}(a([C]))$ is already a $\IP^2$-family of such curves, so there are no others.

If on the other hand $u([C]) = j(y)$ then $C$ is non-CM, and consists of a singular plane cubic curve $C_0 \subset S_C$ together with an embedded point at $y$.  In particular $C$ has only one embedded point, so if two curves $C_1$ and $C_2$ both have embedded points at $y$ then $u([C_1]) = u([C_2]) = j(y)$.

This concludes our recollections from \cite{llsvs}.
\vspace \baselineskip

Let $L_k\colon D(Y) \to \langle \ko_Y(k) \rangle^\perp \subset D(Y)$ be the left mutation past $\ko_Y(k)$:
\[ L_k(B) = \cone\Bigl(\ko_Y(k)\otimes \RHom(\ko_Y(k),B)\to B \Bigr). \]
Then the composition $\pr := L_{-1} \circ L_0 \circ L_1$ is the projection into $\ka$ discussed in the introduction.  It annihilates $\ko_Y(-1)$, $\ko_Y$, and $\ko_Y(1)$, and acts as the identity on $\ka$.  It is left adjoint to the inclusion $\ka \hookrightarrow D(Y)$.

\begin{lemma} \label{lem:projections} \ 
\begin{enumerate}
\item[(a)] For all $[C] \in M$ one has $\pr(I_{C/Y}(2)) \cong \pr(I_{C/S_C}(2))$.
\item[(b)] If $u([C]) = j(y)$ then $\pr(I_{C/Y}(2)) \cong \pr(\ko_y)[-1]$.
\end{enumerate}
\end{lemma}
\begin{proof}
(a) From the Koszul resolution
\[ 0 \to \ko_Y \to \ko_Y(1)^2 \to I_{S_C/Y}(2) \to 0 \]
we see that that $\pr(I_{S_C/Y}(2)) = 0$, so from the exact sequence
\[ 0 \to I_{S_C/Y}(2) \to I_{C/Y}(2) \to I_{C/S_C}(2) \to 0 \]
we see that $\pr(I_{C/Y}(2)) \cong \pr(I_{C/S_C}(2))$.

(b) Let $C_0$ be the singular plane cubic recalled above.  From the Koszul resolution
\[ 0 \to \ko_Y(-1) \to {\ko_Y}^3 \to \ko_Y(1)^3 \to I_{C_0/Y}(2) \to 0 \]
we see that $\pr(I_{C_0/Y}(2)) = 0$, so from the exact sequence
\[ 0 \to I_{C/Y}(2) \to I_{C_0/Y}(2) \to \ko_y \to 0 \]
we see that $\pr(I_{C/Y}(2)) \cong \pr(\ko_y)[-1]$.
\end{proof}

\begin{proposition} \label{prop:fibers}
Two points $[C_1], [C_2] \in M$ lie in the same fiber of $u\colon M \to Z$ if and only if $\pr(I_{C_1}(2)) \cong \pr(I_{C_2}(2))$.
\end{proposition}
\begin{proof}
If $u([C_1]) = u([C_2]) \notin j(Y)$ then $a([C_1]) = a([C_2])$, so $I_{C_1/S_{C_1}} \cong I_{C_2/S_{C_2}}$, so $\pr(I_{C_1}(2)) \cong \pr(I_{C_2}(2))$ by Lemma \ref{lem:projections}(a).  If $u([C_1]) = u([C_2]) = j(y)$ then $\pr(I_{C_1}(2)) \cong \pr(I_{C_2}(2))$ by Lemma \ref{lem:projections}(b).

Conversely, suppose that $\pr(I_{C_1}(2)) \cong \pr(I_{C_2}(2))$.  We consider three cases.

Case 1: $C_1$ and $C_2$ are both aCM.  It is enough to show that $\pr(I_{C/S_C}(2))$ determines $I_{C/S_C}(2)$ for every aCM curve $C$.  From \eqref{eq:acm_res} and \eqref{eq:ICS_vs_ICP} we find that $H^*(I_{C/S_C}) = H^*(I_{C/S_C}(1)) = 0$, so $I_{C/S_C}(2) \in \langle \ko_Y(1), \ko_Y(2) \rangle^\perp$.  Moreover we find that $I_{C/S_C}(2)$ is generated by global sections, so $F_C := L_0(I_{C/S_C}(2))[-1]$ is a sheaf and fits into an exact sequence
\[ 0\to F_C\to \ko_Y^3\to I_{C/S_C}(2) \to 0. \]
As $I_{C/S_C}(2)$ has codimension 2, dualizing this sequence gives $F_C\dual\cong(\ko_Y^3)\dual$, and dualizing again shows that the inclusion of $F_C$ in $\ko_X^3$ is isomorphic to the natural map from $F_C$ to its double dual. Hence $I_{C/S_C}(2)$ can be recovered from $F_C$ as its cotorsion: $I_{C/S_C}(2)\cong F_C\ddual/F_C.$
Now $F_C$ is contained in $\langle \ko_Y,\ko_Y(1),\ko_Y(2)\rangle^\perp$. Since the canonical bundle $\omega_Y$ is $\ko_Y(-3)$, the left mutation $L_{-1}$ and the corresponding right mutation $R_{-1}$ provide inverse equivalences
\[ \xymatrix{
\langle \ko_Y,\ko_Y(1),\ko_Y(2)\rangle^\perp \ar@<.5ex>[r]^{\ L_{-1}\ } &
\langle\ko_Y(-1),\ko_Y,\ko_Y(1)\rangle^\perp  \ar@<.5ex>[l]^{\ R_{-1}\ }
}. \] 
Hence $\pr(I_{C/S}(2))=L_{-1}(F_C)$ determines $F_C$ and hence $I_{C/S}(2)$.

Case 2: $C_1$ is aCM and $C_2$ is non-CM with embedded point $y_2$.  Since $\pr$ is left adjoint to the inclusion $\ka \hookrightarrow D(Y)$, we have
\begin{align*}
\Hom\bigl( \pr(I_{C_1}(2)),\ \pr(I_{C_2}(2)) \bigr)
&= \Hom\bigl( \pr(I_{C_1/S_{C_1}}(2)),\ \pr(\ko_{y_2})[-1] \bigr) \\
&= \Hom\bigl( I_{C_1/S_{C_1}}(2),\ \pr(\ko_{y_2})[-1] \bigr).
\end{align*}
Applying $\pr$ to  $\ko_{y_2}[-1]$ we get a truncated Koszul complex
\begin{equation} \label{eq:trunc_koszul}
\ko_Y(-1)^{10} \to {\ko_Y}^5 \to \underline{\ko_Y(1)} \to \ko_{y_2},
\end{equation}
where the underlined term is in degree zero.  Applying $\Hom(I_{C_1/S_{C_1}}(2),\ -)$ to the complex \eqref{eq:trunc_koszul} we find that the $E_1$ page of the Grothendieck spectral sequence is
\[ \begin{array}{ccccccc|c}
0 & \to & * & \to & * & \to & * \\
0 & \to & * & \to & * & \to & * \\
0 & \to & * & \to & * & \to & * & q=2 \\
0 & \to & 0 & \to & 0 & \to & * & q=1 \\
0 & \to & 0 & \to & 0 & \to & * & q=0 \\
\hline
p=-2 & & p=-1 & & p=0 & & p=1
\end{array} \]
where in the left-hand column we have used the fact that
\[ \Ext^q\bigl( I_{C_1/S_{C_1}}(2),\ \ko_Y(-1) \bigr)
= H^{4-q}\bigl( I_{C_1/S_{C_1}} \bigr)\dual = 0 \]
and the other zeroes come for dimension reasons.  From this it follows that \linebreak $\Hom( I_{C_1/S_{C_1}}(2),\ \pr(\ko_{y_2})[-1] ) = 0$, so $\pr(I_{C_1}(2)) \not\cong \pr(I_{C_2}(2))$.

Case 3: $C_1$ and $C_2$ are non-CM with embedded points $y_1$ and $y_2$.  We have
\begin{align*}
\Hom\bigl( \pr(I_{C_1}(2)),\ \pr(I_{C_1}(2)) \bigr)
&= \Hom\bigl( \pr(\ko_{y_1}),\ \pr(\ko_{y_2}) \bigr) \\
&= \Hom\bigl( \ko_{y_1},\ \pr(\ko_{y_2}) \bigr).
\end{align*}
By a similar Grothendieck spectral sequence calculation, this is $\Hom(\ko_{y_1}, \ko_{y_2})$.  Thus if $\pr(I_{C_1}(2)) \cong \pr(I_{C_1}(2))$ then this $\Hom$ does not vanish, so $y_1 = y_2$, so $u([C_1]) = u([C_2])$.
\end{proof}

Not only are the points of $Z$ in bijection with the objects $\pr(I_C(2))$, but in fact the tangent spaces of $Z$ can be identified with the deformation spaces of the corresponding objects, so $Z$ truly deserves to be called a moduli space of objects in $\ka$.  But we will not prove this, as we do not need it for our main theorem.

Of course one would like to be able to define $Z$ directly as the component of the moduli space of stable objects in $\ka$ containing $\pr(\ko_y)[-1]$, thus avoiding the hard work of \cite{llsvs}.  At present, though, no one knows how to produce any kind of stability condition on $\ka$ when $Y$ is general.  So while the derived perspective clarifies the construction of \cite{llsvs}, it cannot yet replace it.

\section{Pfaffian cubics} \label{sec2}

Let $V$ be a 6-dimensional complex vector space and $L\subset\Lambda^2V^*$ a generic 6-dimensional subspace of skew-symmetric forms on $V$. To these data Beauville and Donagi associate a K3 surface 
\[ X=\Bigl\{[P]\in\Grass(2,V) \Bigm| \varphi|_P=0\text{ for all }\varphi \in L\Bigr\} \]
and a Pfaffian cubic fourfold
\[ Y=\Bigl\{[\varphi]\in \IP(L^*) \Bigm| \rk(\varphi)=4\Bigr\}
=\Bigl\{[\varphi]\in \IP(L^*) \Bigm| \pf(\varphi) = 0\Bigr\}. \]
(Here we use Grothendieck's convention that $\IP(L^*)$ is the space of 1-dimensional quotients of $L^*$, hence of 1-dimensional subspaces of $L$.)  For a generic choice of $L$, both $X$ and $Y$ are smooth, and $X$ does not contain a line nor $Y$ a plane.  Under this genericity assumption, $Y$ cannot contain a quadric surface either, for the linear hull of any quadric surface $Q\subset Y$ would cut out a residual plane.

In this section we study the correspondence
\[ \Gamma = \Bigl\{ ([P],[\varphi]) \in X\times Y \Bigm| P\cap \rad(\varphi) \ne 0 \Bigr\} \]
and show that its ideal sheaf induces an equivalence between $D(X)$ and $\ka$.  This has been proved by Kuznetsov in \cite{kuz_gr} -- see in particular the proof of his Lemma 8.2, where he shows that his kernel is $I_\Gamma$ twisted by a line bundle -- but as his machinery is very heavy we prefer to give a self-contained account.  The fiber of $\Gamma$ over a point of $X$ is a degree-4 ruled surface on $Y$, as we shall see below; this family of surfaces is mentioned by Beauville and Donagi \cite[Rem.~(1)]{bd} and goes back to Fano \cite{fano}.

The correspondence $\Gamma$ carries a natural scheme structure defined as follows: Let $0\to \kp\to V_X\to \kq\to 0$ denote the tautological bundle sequence on $X$, and let $A\colon V_{\IP(L^*)}\to V_{\IP(L^*)}\dual \otimes \ko(1)$ denote the tautological skew-symmetric form parametrized by $\IP(L^*)$. By construction, the restriction of $A_\varphi$ to any $P$, $[P]\in X$, vanishes, so $A$ induces a homomorphism $A'\colon\kp \boxtimes \ko \to \kq\dual \boxtimes \ko(1)$ on $X\times \IP(L^*)$. Then $\Gamma\subset X\times \IP(L^*)$ is the subscheme defined by the vanishing of the $2\times 2$-minors of $A'$. There are natural morphisms $X\xla{\;p_X\;} \Gamma\xra{\;p_Y\;}Y$.  

For any $[\varphi]\in Y$, the radical $\rad(\varphi)$ is a plane in $V$ which, however, can never lie in $X$:  In fact, up to a scalar factor, the differential $D_\varphi\pf$ maps a tangent vector $\psi$ to its value on $\Lambda^2\rad(\varphi)$. As $Y$ is smooth, the intersection of $\IP(L^*)$ and the Pfaffian hypersurface at $[\varphi]$ is transversal. Hence not all $\varphi' \in L$ can vanish on $\rad(\varphi)$.

Thus the fiber $\Gamma_P:=p_X^{-1}([P])$ admits a well-defined map $\pi\colon \Gamma_P\to \IP(P)$, $\varphi\mapsto P\cap \rad(\varphi)$.  The fiber $\pi^{-1}([\ell])$ over a line $\ell\subset P$ is a linear subspace in $Y$.  But by assumption, $Y$ does not contain a plane.  Hence this fiber is at most 1-dimensional, and in turn $\dim(\Gamma_P)\leq 2$ and $\dim(\Gamma) \le 4$. As $\Gamma$ is a determinantal variety, there is the \emph{a priori} bound $\codim(\Gamma/X \times \IP(L^*))\leq 3$.  We conclude that $\Gamma$ has the expected dimension 4, and that the Eagon-Northcott complex associated to $A'{}^t$ is a locally free resolution of the ideal sheaf $I_{\Gamma/X \times \IP(L^*)}$, and $\Gamma$ is Cohen--Macaulay (cf.\ \cite[Thm.~A2.10 and Cor.~A2.13]{eis}).  Restricting the complex to $[P] \times \IP(L^*)$, we obtain a locally free resolution
\begin{equation}\label{eq:Eagon-Northcott}
0\to \ko_{\IP(L^*)}(-4)^3\to \ko_{\IP(L^*)}(-3)^8\to \ko_{\IP(L^*)}(-2)^6
\to \ko_{\IP(L^*)}\to \ko_{\Gamma_P}\to 0.
\end{equation}
In particular, the Hilbert polynomial of $\Gamma_P$ is constant as a function of $P$, and $p_X\colon \Gamma\to X$ is flat. Moreover, each $\Gamma_P$ is a 2-dimensional Cohen-Macaulay subscheme of $Y$ of degree $4$.  Since $Y$ does not contain planes or quadric surfaces, $\Gamma_P$ is generically reduced and hence reduced.
\vspace\baselineskip

Let $\Phi\colon D(Y)\to D(X)$ be the Fourier-Mukai functor induced by the ideal sheaf $I_\Gamma = I_{\Gamma/X\times Y}$, and let $\Psi\colon D(X)\to D(Y)$ be its right adjoint. From the resolution \eqref{eq:Eagon-Northcott} and the exact sequence
\[ 0\to \ko_{\IP(L^*)}(-3)\to I_{\Gamma_P/\IP(L^*)}\to I_{\Gamma_P/Y}\to 0 \]
we find that $\Phi(\ko_Y(k))=0$ for $k=-1,0,1$.  This implies that $\Psi(D(X))\subset\ka$.

\begin{proposition}[Kuznetsov] \label{prop:equiv}
$\Psi\colon D(X)\to \ka$ is an equivalence.
\end{proposition}

\begin{proof} Step 1. For distinct points $[P]$, $[Q]\in X$, the corresponding subvarieties $\Gamma_P$,
$\Gamma_Q\subset Y$ are distinct:  We have $P \cap Q = 0$; otherwise $X$ would contain a line.  Hence if $\Gamma_P=\Gamma_Q$, the mappings $[\varphi]\mapsto \rad(\varphi)\cap P$ and $[\varphi]\mapsto \rad(\varphi)\cap Q$ would define two different rulings on $\Gamma_P=\Gamma_Q$, which is impossible. 

Step 2. The functor $\Psi$ is fully faithful: By the criterion of Bondal and Orlov \cite[
Prop.~7.1]{huybrechts_fm}, it is enough to show that
\begin{equation} \label{eq:bo_criterion}
\dim\Ext^i_Y(\Psi(\ko_{[P]}),\Psi(\ko_{[Q]}))=\dim\Ext^i_X(\ko_{[P]},\ko_{[Q]}).
\end{equation}
The kernel inducing $\Psi$ is $I\dual\otimes \ko_Y(-3)[4]$ (cf.\ \cite[Prop.~5.9]{huybrechts_fm}).  Since $\Gamma$ is flat over $X$ we have $\Psi(\ko_{[P]})=I\dual_{\Gamma_P/Y}(-3)[4]$, so we can rewrite \eqref{eq:bo_criterion} as
\[ \dim\Ext^i_Y(I_{\Gamma_Q/Y},I_{\Gamma_P/Y})=\dim\Ext^i_X(\ko_{[P]},\ko_{[Q]}). \]
This is trivial for $i<0$ and obviously
true for $i=0$, since $\Gamma_P$ and $\Gamma_Q$ have codimension 2 in $Y$ and are distinct if $P\neq Q$. Kuznetsov \cite[Cor.~4.4]{kuz_serre} has shown that the Serre functor of $\ka$ is given by shifting by 2.  Thus Serre duality gives the claim for $i\geq 2$.  Finally, Hirzebruch-Riemann-Roch gives $\chi(I_{\Gamma_Q/Y},I_{\Gamma_P/Y})=0$, so the claim also holds in the remaining case $i=1$.

Step 3. Since $\Psi$ is fully faithful and the Serre functor is given by shifting by 2 on both $D(X)$ and $\ka$, it is enough to show that $\ka$ is indecomposable \cite[Cor.~1.56]{huybrechts_fm}.  This follows from the fact that $HH^0(\ka) \cong HH_{-2}(\ka)$ is 1-dimensional \cite{kuz_hochschild}.
\end{proof}

\begin{lemma} \label{lem:4-1}
The projection $p_Y\colon \Gamma \to Y$ is generically finite of degree 4.
\end{lemma}
\begin{proof}
Fix $[\varphi] \in Y$.  Then the fiber $\Gamma_\varphi := p_Y^{-1}([\varphi])$ is a linear section of the Schubert cycle
\[ \Sigma_\varphi := \Bigl\{ [Q] \in \Grass(2,V) \Bigm| Q \cap \rad(\varphi) \ne 0 \Bigr\}, \]
which is 5-dimensional.  Choose a basis $\varphi_1, \dotsc, \varphi_6$ of $L$ with $\varphi_1 = \varphi$.  This determines 6 hyperplane sections $\varphi_1^\perp, \dotsc, \varphi_6^\perp$ of $\Grass(2,V)$, and $\Gamma_\varphi = \Sigma_\varphi \cap \varphi_1^\perp \cap \dotsb \cap \varphi_6^\perp$.  But $\Sigma_\varphi$ is already contained in $\varphi_1^\perp = \varphi^\perp$, so $\Gamma_\varphi$ is the intersection of $\Sigma_\varphi$ with 5 hyperplanes, hence is non-empty.   Since $\dim \Gamma = \dim Y$, we see that $p_Y$ is generically finite.  With a bit of Schubert calculus we find that $\deg \Sigma_\varphi = 4$, so when $\Gamma_\varphi$ has the expected dimension it is a 0-dimensional scheme of length 4.
\end{proof}
In fact one can show that if $X$ contains no $(-2)$-curves then $p_Y\colon \Gamma \to Y$ is flat, but we do not need this.

\section{The birational isomorphism} \label{sec3}

Now we assemble the results from the previous two sections to prove the theorem stated in the introduction.  As in the previous section, let $Y$ be a Pfaffian cubic and $X$ the associated K3 surface, and assume that $X$ does not contain a line nor $Y$ a plane.  All our pullbacks, pushforwards, etc.\ are implicitly derived.

Let $\kc \subset M \times Y$ be the universal curve, and let $T$ be the convolution
\[ T := I_\Gamma \circ I_\kc(2) = \pi_{M \times X*} \left( \pi_{X \times Y}^* I_\Gamma \otimes \pi_{M \times Y}^* I_\kc(2) \right) \in D(M \times X). \]
For each $[C] \in M$, let $i_{[C]}\colon X \to M \times X$ be the inclusion $x \mapsto ([C],x)$.  Because $\kc$ is flat over $M$, the derived restriction $i_{[C]}^* T$ is isomorphic to $\Phi(I_C(2))$.

By Lemma \ref{lem:4-1}, there is an open subset $Y_0 \subset Y$ such that $\Phi(\ko_y)$ is the ideal sheaf $I_{\xi(y)/X}$ of a 0-dimensional subscheme $\xi(y) \subset X$ of length 4 for all $y \in Y_0$.  Since $\Phi$ annihilates $\ko_Y(-1)$, $\ko_Y$, and $\ko_Y(1)$, we have $\Phi \circ \pr = \Phi$, so by Lemma \ref{lem:projections}(b), the sheaves $\kh^k(i_{[C]}^* T)$ vanish for $k\neq 1$ whenever $u([C]) \in j(Y_0)$. By semicontinuity, the same then holds for all $[C]$ in an open neighborhood $M_0$ of $u^{-1}(j(Y_0))$.  Hence by \cite[Lem.~3.31]{huybrechts_fm}, the sheaf $E:=\kh^1(T|_{M_0\times X})$ is flat over $M_0$, and $T|_{M_0\times X} \cong E[-1]$.  Over $u^{-1}(j(Y_0)) \subset M_0$ the family $E$ parametrizes ideal sheaves on $X$, and since ideal sheaves are stable, we conclude after shrinking $M_0$ if necessary that $E$ is an $M_0$-flat family of ideal sheaves on $X$.

Let $t'\colon M_0\to \Hilb^4(X)$ be the classifying morphism induced by the family $E$.  Proposition \ref{prop:fibers} implies that $t'$ is constant on the fibers of $u$.  As $u$ is proper, there is an open neighborhood $Z_0$ of $j(Y_0)$ in $Z$ such that $u^{-1}(Z_0) \subset M_0$.  The restriction $T'_{u^{-1}(Z_0)}$ now descends to give a morphism $t\colon Z_0\to \Hilb^4(X)$.

It follows from Proposition \ref{prop:equiv} that $\Psi \circ \Phi = \pr$, so by Proposition \ref{prop:fibers} again we see that $t$ is injective.  The differential of $t$ must have full rank at some point --- otherwise $t(Z_0)$ would be a proper subscheme of $\Hilb^4(X)$, contradicting injectivity --- and hence it must have full rank on an open subset of $Z_1$ of $Z_0$.  Now $t|_{Z_1}$ is injective and \'etale, hence is an open immersion.  Thus $Z$ is birational to $\Hilb^4(X)$.

\bibliographystyle{plain}
\bibliography{main}

\begin{thebibliography}{10}

\bibitem{bd}
A.~Beauville and R.~Donagi.
\newblock La vari\'et\'e des droites d'une hypersurface cubique de dimension 4.
\newblock {\em C. R. Acad. Sci. Paris S\'er. I Math.}, 301(14):703--706, 1985.
\newblock Also
  \href{http://math1.unice.fr/~beauvill/pubs/bd.pdf}{\nolinkurl{math1.unice.fr/~beauvill/pubs/bd.pdf}}.

\bibitem{eis}
D.~Eisenbud.
\newblock {\em Commutative algebra}, volume 150 of {\em Graduate Texts in
  Mathematics}.
\newblock Springer-Verlag, New York, 1995.
\newblock With a view toward algebraic geometry.

\bibitem{fano}
G.~Fano.
\newblock Sulle forme cubiche dello spazio a cinque dimensioni contenenti
  rigate razionali del {$4^\circ$} ordine.
\newblock {\em Comment. Math. Helv.}, 15:71--80, 1943.

\bibitem{huybrechts}
D.~Huybrechts.
\newblock Compact hyper-{K}\"ahler manifolds: basic results.
\newblock {\em Invent. Math.}, 135(1):63--113, 1999.
\newblock Also \href{http://arxiv.org/abs/alg-geom/9705025}{alg-geom/9705025}.

\bibitem{huybrechts_fm}
D.~Huybrechts.
\newblock {\em Fourier-{M}ukai transforms in algebraic geometry}.
\newblock Oxford Mathematical Monographs. The Clarendon Press Oxford University
  Press, Oxford, 2006.

\bibitem{kuz_hochschild}
A.~Kuznetsov.
\newblock Hochschild homology and semiorthogonal decompositions.
\newblock Preprint, \href{http://arxiv.org/abs/0904.4330}{arXiv:0904.4330}.

\bibitem{kuz_gr}
A.~Kuznetsov.
\newblock Homological projective duality for {G}rassmannians of lines.
\newblock Preprint, \href{http://arxiv.org/abs/math/0610957}{math/0610957}.

\bibitem{kuz_serre}
A.~Kuznetsov.
\newblock Derived category of a cubic threefold and the variety {$V_{14}$}.
\newblock {\em Tr. Mat. Inst. Steklova}, 246(Algebr. Geom. Metody, Svyazi i
  Prilozh.):183--207, 2004.
\newblock Also \href{http://arxiv.org/abs/math/0303037}{math/0303037}.

\bibitem{km}
A.~Kuznetsov and D.~Markushevich.
\newblock Symplectic structures on moduli spaces of sheaves via the {A}tiyah
  class.
\newblock {\em J. Geom. Phys.}, 59(7):843--860, 2009.
\newblock Also \href{http://arxiv.org/abs/math/0703264}{math/0703264}.

\bibitem{llsvs}
C.~Lehn, M.~Lehn, C.~Sorger, and D.~van Straten.
\newblock Twisted cubics on cubic fourfolds.
\newblock {\em J. reine angew. Math.}, to appear.
\newblock Also \href{http://arxiv.org/abs/1305.0178}{arXiv:1305.0178}.

\bibitem{ouchi}
G.~Ouchi.
\newblock Lagrangian embeddings of cubic fourfolds containing a plane.
\newblock Preprint, \href{http://arxiv.org/abs/1407.7265}{arXiv:1407.7265}.

\end{thebibliography}

\end{document}